\documentclass[reqno,a4paper,11pt]{amsart}
\usepackage{color}
\usepackage{amsmath,amssymb,amsfonts,amsthm,enumerate}
\usepackage{hyperref,mathrsfs,accents}
\usepackage[T1]{fontenc}
\usepackage[active]{srcltx}
\usepackage{epsfig}
\usepackage{graphicx}
\usepackage{cite}
\usepackage{tikz}
\usepackage{pgfplots}
\usepackage{tikz-3dplot}
\usepackage{subcaption}
\usepackage{pgfplots}
\usepackage{adjustbox}
\usepackage{bigints}
\usetikzlibrary{patterns,backgrounds}

\tdplotsetmaincoords{60}{115}
\pgfplotsset{compat=newest}

\catcode`@=11 \@addtoreset{equation}{section} \catcode`@=12

\newtheorem{theorem}{Theorem}[section]
\newtheorem{lemma}[theorem]{Lemma}

\newcommand{\R}{\mathbb{R}}

\newcommand{\lowD}{\underline{\de_{t}}}
\newcommand{\lowDd}{\underline{\de_{t}^{2}}}
\newcommand{\lowDs}{\underline{\de_{s}}}

\newcommand{\e}{\varepsilon}
\newcommand{\C}{{\mathcal C}}

\newcommand{\Hau}{{\mathcal H}} 
\renewcommand{\dim}{\mathop{\rm dim}} 
\newcommand{\de}{\partial}

\newcommand{\spt}{\mathop{\mathrm{spt}}}

\newcommand{\cA}{\mathcal{A}}

\newcommand{\cI}{\mathcal{I}}
\newcommand{\cJ}{\mathcal J}

\newcommand{\Lip}{\mathrm{Lip}}

\AtBeginDocument{

}

\definecolor{grey}{rgb}{.7,.7,.7}
\definecolor{evidGP}{rgb}{0,0,1}
\definecolor{evidG}{rgb}{1,0,0}

\author{Gian Paolo Leonardi}
\address{Dipartimento di Matematica, via Sommarive 14, IT-38123 Povo - Trento (Italy)}
\email{gianpaolo.leonardi@unitn.it}

\author{Giacomo Vianello}
\address{Dipartimento di Matematica Tullio Levi Civita, Via Trieste 63, IT-35121 Padova (Italy)}
\email{giacomo.vianello@unipd.it}

\thanks{G.P.Leonardi has been partially supported by: PRIN 2017TEXA3H ``Gradient flows, Optimal Transport and Metric Measure Structures''; PRIN 2022PJ9EFL ``Geometric Measure Theory: Structure of Singular Measures, Regularity Theory and Applications in the Calculus of Variations'' (financed by European Union - Next Generation EU, Mission 4, Component 2 - CUP:E53D23005860006); Grant PID2020-118180GB-I00 ``Geometric Variational Problems''. Giacomo Vianello has been supported by GNAMPA (INdAM) Project 2023: ``Esistenza e propriet\`a fini di forme ottime''; GNAMPA (INdAM) Project 2025: ``Structures of sub-Riemannian hypersurfaces in Heisenberg groups''.}



\title[Stability of axial free-boundary hyperplanes...]{Stability of axial free-boundary hyperplanes in circular cones}

\begin{document}

\begin{abstract}
	Given an axially-symmetric, $(n+1)$-dimensional convex cone $\Omega\subset \R^{n+1}$, we study the stability of the free-boundary minimal surface $\Sigma$ obtained by intersecting $\Omega$ with a $n$-plane that contains the axis of $\Omega$. In the case $n=2$, $\Sigma$ is always unstable, as a special case of the vertex-skipping property that we recently proved in another article. Conversely, as soon as $n\ge 3$ and $\Omega$ has a sufficiently large aperture (depending on the dimension $n$), we show that $\Sigma$ is strictly stable. For our stability analysis, we introduce a Lipschitz flow $\Sigma_{t}[f]$ of deformations of $\Sigma$ associated with a compactly-supported, scalar deformation field $f$, which satisfies the key property $\de \Sigma_{t}[f] \subset \de\Omega$ for all $t\in \R$. Then, we compute the lower-right second variation of the area of $\Sigma$ along the flow, and ultimately show that it is positive by exploiting its connection with a functional inequality studied in the context of reaction-diffusion problems. 
\end{abstract}
\maketitle

\section{Introduction}
One of the most intriguing aspects of the theory of free-boundary minimal surfaces, and more generally of capillary surfaces, is their behaviour near the boundary of their container. A well-established regularity theory is known to hold when the free boundary of a minimal surface $\Sigma$ meets a regular point of the boundary of its container $\Omega$, see \cite{DePhilippisMaggi2015,GruterJost1986allard}; in this case the $90$-degree contact-angle law holds true in a classical sense. More generally, Young's law for capillary surfaces \cite{young1805iii,DePhilippisMaggi2015} holds. Conversely, the behaviour of $\Sigma$ near a generic point of $\de\Sigma \cap \de^{s}\Omega$, where $\de^{s}\Omega$ denotes the set of singular points of $\de\Omega$, is much less known. Few exceptions are represented by the studies on free-boundary minimal surfaces in $3$-dimensional wedge domains, mainly due to Hildebrandt-Sauvigny \cite{hildebrandt1999wedge4,hildebrandt1999wedge3,hildebrandt1997wedge2,hildebrandt1997wedge1}, and by the analogous studies pursued in the context of capillary surfaces, see e.g. \cite{Fin86book,Lancaster2012} and the references therein. 

A few years ago, Edelen-Li \cite{EdelenLi2022} extended the existing regularity theory by proving a very general $\e$-regularity theorem for free-boundary integral varifolds with bounded mean curvature in locally-almost-polyhedral domains. Subsequently, with the aim of pushing the regularity theory forward and beyond the locally almost-polyhedral class, in a couple of recent works we have shown some sort of complementary results in the framework of sets of finite perimeter. The first result is a free-boundary monotonicity formula for perimeter almost-minimizers in a non-smooth domain which satisfies a so-called \textit{visibility condition} \cite{LeoVia2-2024}. The second is a \textit{vertex-skipping property} for almost-minimal boundaries in convex sets \cite{LeoVia1-2024}, which can be coupled with \cite{EdelenLi2022} and leads to a full, free-boundary regularity result for $2$-dimensional integral varifolds with bounded mean curvature in locally almost-polyhedral $3d$-domains.

Blow-up techniques are among the main tools used to retrieve information on the boundary properties of $\Sigma$, such as local structure, density properties, and ultimately regularity. Relying on boundary monotonicity formulas (see, e.g., \cite{LeoVia2-2024}), one can blow-up an almost-minimal surface at one of its free-boundary points, obtaining free-boundary cones inside conical domains. This represents the free-boundary counterpart of classical monotonicity and blow-up results that hold at interior points \cite{Giu84book,maggi2012sets}. By relying on such results, we have shown in \cite{LeoVia1-2024} that the internal boundary $\Sigma$ of a perimeter almost-minimizer in a convex set $\Omega\subset \R^{3}$ avoids all vertices of $\de\Omega$. Here, by vertex we mean a point $x\in \de\Omega$ such that the tangent cone to $\Omega$ at $x$ does not contain full lines. It is then natural to ask whether this \textit{vertex-skipping property} is a special feature of the $3$-dimensional setting, or if instead it holds in any dimension. We remark that, in dimension $2$, the property is trivially satisfied and, for this reason, this elementary case will not be discussed. 

Here, we consider a $n$-plane in $\R^{n+1}$, $n\ge 2$, containing the axis of a convex circular $(n+1)$-cone, and study the stability of the free-boundary minimal surface obtained by intersecting the plane with the cone. More precisely, we take $0<\lambda<\infty$, then we let $x' = (x_{1},\dots,x_{n-1})$ and set
\begin{equation*}
	\Omega_{\lambda} := \left \{ (x',x_{n},x_{n+1}) \in \R^{n+1} \, : \, x_{n} > \omega_{\lambda}(x',x_{n+1}) \right \} \, ,
\end{equation*}
where $\omega_{\lambda}(x',x_{n+1}) = \lambda \sqrt{|x'|^{2} + x_{n+1}^2}$. For future reference, we call $\lambda$ the aperture parameter of the cone, while the angle $\alpha_{\lambda} \in (0,\pi)$ such that $\lambda = \cot (\alpha_{\lambda}/2)$ will be referred to as the \textit{aperture} of the cone. 
Then, we define
\[
\Sigma = \{x\in \Omega_{\lambda}:\ x_{n+1} = 0\}\,.
\]
Section \ref{sec:main} contains the main results of the paper. Specifically,  Theorem \ref{thm:d+<0} shows that, when $n=2$, $\Sigma$ is unstable in $\Omega_{\lambda}$ for every $0 < \lambda < \infty$. We highlight that this is exactly the behaviour one should expect according to the vertex-skipping property proved in \cite{LeoVia1-2024}. More interestingly, Theorem \ref{thm:stab4} establishes that, when $n\ge3$, there exists a threshold aperture parameter $0 < \lambda^{\ast} < \infty$ depending on the dimension $n$, such that, for all $0<\lambda\le \lambda^{*}$ (i.e, for all apertures $\alpha \in [\alpha_{\lambda^{*}},\pi)$), the $n$-plane $\Sigma$ is a strictly stable free-boundary minimal surface in $\Omega_{\lambda}$. The proof of this theorem relies on two main tools. The first is the construction of a suitable Lipschitz flow of deformations that preserves the free-boundary condition and allows a direct writing of first and second variation (in)equalities, see Section \ref{sec:setting}. The second is a variational inequality, called \textit{Kato's inequality}, which is the key to show stability under appropriate conditions on the dimension $n$ and the aperture of $\Omega_{\lambda}$, see Lemma \ref{lem:Kato}. We point out that such inequality emerges in fairly different contexts, from spectral theory \cite{Herbst1977} to reaction-diffusion problems \cite{DavilaDupaigneMontenegro2008}. 

The stability versus instability properties of $\Sigma$ exhibit a close relationship with analogous properties of minimal surfaces within cones. In particular, a result of Morgan \cite{morgan2002area} implies that the free-boundary $\de\Sigma$ is an area-minimizing $(n-1)$-surface in $\de\Omega_{\lambda}$ as soon as $n\ge 4$ and $\lambda$ is small enough. Similar results are also available for free-boundary problems of Alt-Caffarelli type \cite{Allen2017}, where the dimension $n=3$ is critical for the free-boundary to contain the vertex of $\Omega_{\lambda}$, again when $\lambda$ is small enough.

Going back to \cite{morgan2002area}, we note that its main result is based on a calibration argument that holds precisely when $n\ge 4$. For this reason, we expect that our strict stability result can be improved to a full area-minimizing result at least when $n\ge 4$. However, the same kind of analysis in the case $n=3$ seems a harder open problem, as it is not fully clear whether $\de \Sigma$ is unstable, or at least not area-minimizing, in $\de \Omega_{\lambda}$. Our Theorem \ref{thm:stab4} seems to support the conjecture that $\Sigma$ is locally area-minimizing also when $n = 3$. The issue of the area-minimization property when $n\ge 3$ will be addressed in a forthcoming work.
\bigskip


\section{Notations and basic constructions}
\label{sec:setting}
Let $n\ge 2$ be a given integer. Given $x\in \R^{n}$ and $t\in \R$, we define $\tilde{x} = (x,t) \in \R^{n+1}$ and $x'\in \R^{n-1}$ such that $x = (x',x_{n})$, with a slight abuse of notation. We identify $x$ with $(x,0)\in \R^{n+1}$ and $x'$ with $(x',0,0)\in \R^{n+1}$, whenever this does not create confusion. We denote by $\tilde{B}_r(\tilde{x})$ the open ball of radius $r > 0$ centered at $\tilde{x}$, and we set $\tilde{B}_r := \tilde{B}_r(0)$. Given a differentiable function $f:\R^{n+1}\to \R$, we denote the partial derivative of $f$ with respect to $x_{i}$ as $\de_i f$.

For any $1 \leq i \leq n+1$, we denote by $p_i$ the hyperplane of $\R^{n+1}$ of equation $x_i = 0$, and use the same notation for the orthogonal projection of $\R^{n+1}$ onto this hyperplane with a slight abuse of notation.

\subsection{Foliating a Lipschitz epigraph} 
Given a Lipschitz function $\omega = \omega(x',t) : \R^n \to \R$, we consider the Lipschitz epigraph 
\begin{equation} \label{eq:omegagraph}
	\Omega = \{ \tilde{x} = (x',x_{n},t) \in \R^{n+1} \, : \, x_{n} > \omega(x',t) \} \, .
\end{equation}
We also let $\Sigma := \Omega \cap p_{n+1}$ and set $\de\Sigma = \overline{\Sigma}\cap \de\Omega$.
Given $x=(x',x_{n}) \in \overline{\Sigma}$, we define the parametric curve $\Gamma_{x}:\R\to \R^{n+1}$ as
\begin{align} \label{eq:Gamma}
	\Gamma_{x}(t) & := (x',x_n + \omega(x',t) - \omega(x',0),t) 
\end{align}
and, with a slight abuse of notation, we identify $\Gamma_{x}$ with $\Gamma_{x}(\R)$. 

The next lemma collects some key properties of the family $\{\Gamma_{x}:\ x\in \overline\Sigma\}$.
\begin{lemma} \label{lem:propGamma}
	The family of parametric curves $\{\Gamma_{x}:\ x\in \overline\Sigma\}$ defines a foliation of $\overline{\Omega}$. More precisely, the following properties hold:
	\begin{itemize}
		\item[(i)] $\Gamma_{x}(0) = x$ and $\Gamma_{x} \subset \overline\Omega$, for all $x \in \overline{\Sigma}$;
		\item[(ii)] if $x\neq y \in \overline{\Sigma}$ then $\Gamma_{x}\cap \Gamma_y = \emptyset$;
		\item[(iii)] $\Gamma_{x}(t) \in \de\Omega$ for some $t\in \R$ if and only if $\Gamma_{x}\subset \de \Omega$.
		\item[(iv)] the map $(x,t)\mapsto \Gamma_{x}(t)$ is Lipschitz.
	\end{itemize}
\end{lemma}
\begin{proof}
	Property (i) directly follows from the definition of $\Gamma_{x}$ and by observing that $x \in \overline{\Sigma}$ implies $x_n \geq \omega(x',0)$, hence
	\begin{equation*}
		x_n + \omega(x',t) - \omega(x',0) \geq \omega(x',t) \, .
	\end{equation*}
	Then, concerning property (ii), if $\Gamma_x(t) = \Gamma_y(u)$ for some $x,y\in \overline\Sigma$ and some $t,u\in \R$, by definition we must have $t=u$, $x'=y'$, and 
	\begin{equation*}
		x_{n} + (\omega(x',t) - \omega(x',0)) = y_{n} + (\omega(y',u) - \omega(y',0)) \,,
	\end{equation*}
	which implies $x_{n}=y_{n}$ and hence $x=y$, which proves (ii). The less obvious implication of (iii) is the only if part: if $\Gamma_x(t) \in \de \Omega$ for some $t\in \R$, then 
	\[ 
	x_{n} + \omega(x',t)) - \omega(x',0)) = \omega(x',t)\,,
	\]
	i.e., $x_{n} = \omega(x',0)$. Consequently, 
	\[
	\Gamma_{x}(u) = (x',\omega(x',0) + \omega(x',u)-\omega(x',0), u) = (x',\omega(x',u), u) \in \de\Omega
	\]
	for all $u\in \R$. Finally, the proof of (iv) easily follows from the estimate
	\begin{align*}
		|\Gamma_{x}(t) - \Gamma_{y}(u)| &\le |x'-y'| + |x_{n}-y_{n}| + |\omega(x',t) - \omega(y',u)| +|\omega(x',0) +\omega(y',0)| + |t-u|\\
		&\le (1+2\Lip(\omega))(|x'-y'| + |x_{n}-y_{n}|+ |t-u|).
	\end{align*} 
\end{proof}

\subsection{A flow of compact deformations for $\Sigma$.}
Let now $f:\overline{\Sigma}\to \R$ be a Lipschitz function whose support $\spt(f)$ is compact in $\overline{\Sigma}$. We define the \emph{flow associated to $f$} as
\begin{equation} \label{eq:Phi}
	\Phi[f](x,t) := \Gamma_{x}(t f(x))
\end{equation}
for all $(x,t) \in \overline{\Sigma}\times \R$. The map $\Phi[f]$ satisfies $\Phi[f](x,t) = x$ and, thanks to Lemma \ref{lem:propGamma}, it is Lipschitz and one-to-one. When $\Phi[f](\cdot,t)$ is differentiable at $x$ (which is true for almost all $x\in \Sigma$), we denote by 
\[
D \Phi[f](x,t) = \left(\de_{1}\Phi[f](x,t),\dots,\de_{n}\Phi[f](x,t)\right)
\]
the Jacobian matrix of $\Phi[f](\cdot,t)$ at $x$, and by $D \Phi[f](x,t)^{T}$ its transpose. We also define
\[
S[f](x,t) := D \Phi[f](x,t)^T \cdot D \Phi[f](x,t)
\]
and
\[
J[f](x,t) := \sqrt{\det(S[f](x,t))} \,,
\]
and note for future reference that
\begin{equation} \label{eq:idJA}
	J[f](x,0) = 1\quad \forall\,x\in \Sigma\, .
\end{equation}
Then, we define 
\[
\Sigma_{t}[f] = \Phi[f](\Sigma,t)
\]
and
\begin{equation} \label{eq:area_Jac}
	A[f](t) := \Hau^{n}(\Phi[f](\Sigma \cap \spt(f),t)) = \int_{\spt(f)} J[f](x,t) \, dx\,,
\end{equation}
where the second identity follows from the Area Formula (see, e.g., \cite[Theorem 2.91]{AmbrosioFuscoPallara2000}). The function $A[f](t)$ represents the area of the compact portion of the hyperplane $\Sigma$ deformed via the flow map $\Phi[f](\cdot,t)$. The minimality/stability of $\Sigma$ is related to the asymptotic properties of $A[f](t)$ when $t\to 0$. In general, we cannot expect that $A[f](t)$ is (twice) differentiable at $t=0$, hence it will not be possible to compute the first and second variations of the area in a classical sense. However, we can test the local minimality of $\Sigma$ by taking lower right variations, i.e., 
\[
\lowD A[f](0^{+}) := \liminf_{t\to 0^{+}}\frac{A[f](t) - A[f](0)}{t}
\]
and, assuming $\lowD A[f](0^{+})=0$,
\[
\lowDd A[f](0^{+}) := 2 \liminf_{t\to 0^{+}}\frac{A[f](t) - A[f](0)}{t^{2}}\,.
\]
Therefore, a first-order necessary condition for local minimality is 
\begin{equation}\label{eq:FONC}
	\lowD A[f](0^{+}) \ge 0\,,\qquad \forall\, f\in C^{0,1}_{c}(\overline\Sigma)\,,
\end{equation}
with strict inequality only if the local minimality is strict. 
Then, a second-order necessary condition is 
\begin{equation}\label{eq:SONC}
	\lowD A[f](0^{+})=0\ \implies\ \lowDd A[f](0^{+}) \ge 0\,,
\end{equation}
for all $f\in C^{0,1}_{c}(\overline\Sigma)$.
\bigskip

\section{Stability of $\Sigma$ in $\Omega_{\lambda}$}
\label{sec:main}
We recall that $\Omega_{\lambda}$ is the $(n+1)$-dimensional circular cone defined as the epigraph of the function
\begin{equation*}
	\omega_{\lambda}(x',x_{n+1}) := \lambda \, \sqrt{|x'|^2 + x_{n+1}^2} \, ,
\end{equation*}
with $\lambda>0$ a fixed parameter defining the aperture of the cone. In this case, for all $f \in C^{0,1}_c(\overline{\Sigma})$, the one-parameter flow associated with $f$ is given by
\begin{align*}
	\Phi[f](x,t) &:= \left(x' \, , \, \omega_{\lambda}(x',tf(x)) + x_{n} - \omega_{\lambda}(x',0) \, , \, t \, f(x) \right)\\ 
	&= \left(x' \, , \, \lambda \, \sqrt{|x'|^{2} + t^2 f(x)^{2}} + x_{n} - \lambda |x'| \, , \, t \, f(x) \right) \, .
\end{align*}
It is immediate to check that $\Phi[f](x,t)$ is differentiable on the set
\begin{equation*}
	\Delta_f := \{ x \in \spt(f) \, : \, x \neq 0 \, , \, f \text{ is differentiable at $x$} \} \cup \left( \Sigma \setminus \spt(f) \right) \, .
\end{equation*}
In particular, since $f$ is Lipschitz, we infer that
\begin{equation*}
	\Hau^n(\Sigma \setminus \Delta_{f}) = 0 \, .
\end{equation*}
For almost all $x \in \Delta_f$ and for $1 \leq i \leq n-1$, we have
\begin{align} \label{eq:de_i}
	\de_i \Phi[f](x,t) & = e_i + \left[ t \, \de_i f(x) \, \de_{n+1}\omega_{\lambda}(x',t f(x)) + \de_i \omega_{\lambda} (x',t f(x)) - \de_i \omega_{\lambda}(x',0) \right] e_{n} + t \, \de_i f (x) e_{n+1} \nonumber \\
	& = e_i + \lambda \left( \dfrac{t^2 f(x) \, \de_i f (x)}{\sqrt{|x'|^2 + t^2 f(x)^2}} + \dfrac{x_i}{\sqrt{|x'|^2 + t^2 f(x)^2}} - \dfrac{x_i}{|x|} \right) e_{n} + t \, \de_i f(x) e_{n+1}
\end{align}
while, if $i = n$,
\begin{align} \label{eq:de_n}
	\de_n \Phi[f](x,t) & = \left( t \, \de_n f (x) \, \de_{n+1}\omega_{\lambda} (x',tf(x)) + 1 \right) e_{n} + t \, \de_n f (x) e_{n+1} \nonumber \\
	& = \left( \lambda \dfrac{t^2 f(x) \, \de_{n}f (x)}{\sqrt{|x'|^2 + t^2 f(x)^2}} + 1 \right) e_{n} + t \, \de_n f(x) e_{n+1} \, .
\end{align}
We introduce the following notations:
\begin{align} \label{eq:A_i}
	& \alpha_i(x,t) := 
	\begin{cases}
		\lambda \left( \dfrac{t^2 f(x) \, \de_i f (x)}{\sqrt{|x'|^2 + t^2 f(x)^2}} + \dfrac{x_i}{\sqrt{|x'|^2 + t^2 f(x)^2}} - \dfrac{x_i}{|x|} \right) & \text{if $i = 1,...,n-1$} \\[10pt]
		\lambda \, \dfrac{t^2 f(x) \de_{n}f(x)}{\sqrt{|x'|^2 + t^2 f(x)^2}} & \text{if $i = n \, ,$}
	\end{cases}
	\\[10pt] \label{eq:B_i}
	& \beta_{i}(x,t) := t \, \de_i f (x) \, , \qquad \forall i = 1,...,n \, .
\end{align}
We thus have
\begin{equation} \label{eq:expr_de}
	\de_i \Phi[f](x,t) = e_i + \alpha_i(x,t) \, e_{n} + \beta_i(x,t) \, e_{n+1} \, , \qquad \forall i = 1,...,n \, .
\end{equation}
In the next lemma we compute the Jacobian $J[f](x,t)$ of the flow map.
\begin{lemma} \label{lem:J^2}
	Given $f \in \C^{0,1}_c(\overline{\Sigma})$, for all $t \in \R$, we have
	\begin{equation} \label{eq:J^2}
		J[f](x,t)^2 = 1 + t^2 \left( |\nabla f(x)|^2 + \dfrac{2 \, \lambda \, f(x) \, \de_{n} f (x)}{\sqrt{|x'|^2 + t^2 f(x)^2 }} \right) + R[f](x,t) \, ,
	\end{equation}
	where $R[f](x,t)$ is such that
	\begin{equation} \label{eq:R_tprop}
		\spt(R[f](t,\cdot)) \subset \spt(f) \, , \qquad \lim_{t \rightarrow 0} \dfrac{R[f](x,t)}{t^2} = 0 \, , \, \, \text{a.e. in $\Sigma \, ,$} \qquad \left| \dfrac{R[f](x,t)}{t^2} \right| \leq C \, ,
	\end{equation}
	where $C \geq 0$ is a constant depending on $\Lip(f)$ only.
\end{lemma}

\begin{proof}
	By Lemma \ref{lem:det} below, we have
	\begin{equation} \label{eq:det}
		J[f]^2 = \left( 1 + \alpha_n \right)^2 \left( 1 + \sum_{i=1}^{n-1} \beta_i^2 \right) + \beta_n^2 \left( 1 + \sum_{i=1}^{n-1} \alpha_i^2 \right) - 2 \, (1 + \alpha_n) \beta_n \sum_{i=1}^{n-1} \alpha_i \beta_i \, .
	\end{equation}
	where $\alpha_i,\beta_{i}$ are defined in \eqref{eq:A_i} and \eqref{eq:B_i}.
	We observe that, for every $1 \leq i \leq n-1$,
	\begin{equation} \label{eq:propA_i}
		\lim_{t \to 0^+} \alpha_i(t) = 0 \qquad \text{for $x \neq 0$.}
	\end{equation}
	In addition,
	\begin{equation} \label{eq:proplim}
		\lim_{t \to 0} \dfrac{\alpha_n[f](t)}{t} = 0 \quad \text{for $x \neq 0$,} \qquad \lim_{t \to 0} \beta_i[f](t) = 0 \quad \text{for all $x$,}
	\end{equation}
	and
	\begin{equation} \label{eq:propbound}
		\dfrac{\|\alpha_n[f](t)\|_{L^{\infty}(\Sigma)}}{t} \leq \lambda \|\de_n f\|_{L^{\infty}(\Sigma)} \, , \qquad \dfrac{\|\beta_i[f](t)\|_{L^{\infty}(\Sigma)}}{t} \leq \|\de_i f\|_{L^{\infty}(\Sigma)} \, .
	\end{equation}
	Let us define
	\begin{equation*}
		R[f] := \alpha_n^2 \left( 1 + \sum_{i=1}^{n-1} \beta_i^2 \right) + 2 \alpha_n \sum_{i=1}^{n-1} \beta_i^2 + \beta_n^2 \sum_{i=1}^{n-1} \alpha_i^2 - 2 \, (1 + \alpha_n) \beta_n \sum_{i=1}^{n-1} \alpha_i \beta_i \, .
	\end{equation*}
	By the definition of $\alpha_i$, $\beta_i$ provided in \eqref{eq:A_i}, it is evident that $\alpha_i$ and $\beta_i$ vanish as $f$ vanishes, and this implies that $\spt(R[f](t,\cdot)) \subset \spt(f)$. Owing to \eqref{eq:propA_i}, \eqref{eq:proplim}, we infer that
	\begin{equation*}
		\lim_{t \rightarrow 0} \dfrac{R[f](x,t)}{t^2} = 0 \, , \, \, \text{a.e. in $\Sigma \, ,$}
	\end{equation*}
	while \eqref{eq:propbound} yields
	\begin{equation*}
		\left| \dfrac{R[f](x,t)}{t^2} \right| \leq C \, , \qquad \text{$C$ constant depending on $\Lip(f)$.}
	\end{equation*}
	This proves the validity of \eqref{eq:R_tprop}. On the other hand
	\begin{align*}
		J[f]^2 - R[f] \ =\ 1 + 2 \alpha_n + \sum_{i=1}^{n} \beta_i^2 
		\ =\ 1 + t^2 \left( |\nabla f(x)|^2 + \dfrac{2 \, \lambda \, f(x) \, \de_{n} f (x)}{\sqrt{|x'|^2 + t^2 f(x)^2 }} \right) \, ,
	\end{align*}
	that is precisely what we claimed.
\end{proof}

Before proving our technical Lemma \ref{lem:det}, we introduce
some convenient notation. We let $\mu = (\mu_{1},\dots,\mu_{k})$ be a multi-index with $\mu_{\ell}\in \{1,\dots,n+1\}$, then we set $e_{\mu} = e_{\mu_{1}}\wedge\dots\wedge e_{\mu_{k}}$ and, given $i\in \{1,\dots,k\}$ and $j\in \{1,\dots,n+1\}$, we define
\[
e_{\mu,i}^{j} = e_{\mu_{1}}\wedge\dots\wedge e_{\mu_{i-1}}\wedge e_{j}\wedge e_{\mu_{i+1}}\wedge\dots\wedge e_{\mu_{k}} \, .
\]
We have the following
\begin{lemma} \label{lem:det}
	Assume that $\de_i \Phi[f]$ are given by \eqref{eq:expr_de}, for $1 \leq i \leq n$. Then $J[f]^2$ is expressed by formula \eqref{eq:det}.
\end{lemma}
\begin{proof}
	For $1 \leq i \leq n$, let $v_i := \de_i \Phi[f]$. Well-known results of multi-linear algebra \cite[Chapter 2]{Shafarevich_Remizov_book} guarantee that
	\begin{equation} \label{eq:normwedge}
		J[f]^2 = |v_1 \wedge \dots \wedge v_n|^2 \, ,
	\end{equation}
	where $|\cdot|$ denotes the standard Euclidean norm for multivectors. 
	Let us set $\mu = (1,2,\dots,n-1)$. 
	By the definition of $v_n$ and the fact that $e^{n}_{\mu,i} \wedge e_{n+1} = - e^{n+1}_{\mu,i} \wedge e_n$, we get
	\begin{align} \label{eq:wedge}
		v_1 \wedge ... \wedge v_n & =  \left( 1 + \alpha_n \right) e_{\mu}\wedge e_{n} + \beta_n \, e_{\mu}\wedge e_{n+1} \, + \nonumber \\
		& \qquad \qquad + \beta_n \sum_{i=1}^{n-1} \alpha_i \, e^{n}_{\mu,i} \wedge e_{n+1} + \left( 1 + \alpha_n \right) \sum_{i=1}^{n-1} \beta_i \, e^{n+1}_{\mu,i} \wedge e_n \\
		& = \left( 1 + \alpha_n \right) e_{\mu}\wedge e_{n} + \beta_n \, e_{\mu}\wedge e_{n+1} + \sum_{i=1}^{n-1} \left[ \beta_n \alpha_i - \left( 1 + \alpha_n \right) \beta_i \right] \, e^{n}_{\mu,i} \wedge e_{n+1} \, . 
		\nonumber
	\end{align}
	By the orthonormality of the multivectors $e_{\mu}\wedge e_{n}, e_{\mu}\wedge e_{n+1}$, and $e^{n}_{\mu,i} \wedge e_{n+1}$ for $i=1,\dots,n-1$, we obtain
	\begin{equation*}
		|v_1 \wedge ... \wedge v_n|^2 = \left( 1 + \alpha_n \right)^2 \left( 1 + \sum_{i=1}^{n-1} \beta_i^2 \right) + \beta_n^2 \left( 1 + \sum_{i=1}^{n-1} \alpha_i^2 \right) - 2 \, (1 + \alpha_n) \beta_n \sum_{i=1}^{n-1} \alpha_i \beta_i \, .
	\end{equation*}
	Then \eqref{eq:det} immediately follows from \eqref{eq:normwedge}.
\end{proof}

We observe that \eqref{eq:J^2} explicitly depends on $t^{2}$. This implies the vanishing of the first (lower) variation of the area, $\lowD A[f](0^{+}) = 0$, and the fact that the second variation can be computed as the first variation of the area with respect to the parameter $s= t^{2}$. To this end, given $s > 0$ we conveniently set $\cA[f](s) = A[f](\sqrt s)$ and $\cJ[f](x,s) = J[f](x,\sqrt s)$. The following theorem holds.

\begin{theorem} \label{thm:d+}
	The following identity holds:
	\begin{equation} \label{eq:d+A}
		\lowDs \cA[f](0^{+}) \, = \, \dfrac{1}{2} \left( \int_{\Sigma} |\nabla f(x)|^2 \, dx - \lambda \int_{\R^{n-1}} \dfrac{f(x',\lambda|x'|)^2}{|x'|} \, d x' \right) \, , \qquad \forall f \in C^{0,1}_c(\overline{\Sigma}) \, .
	\end{equation}
\end{theorem}

\begin{proof}
	Owing to \eqref{eq:idJA}, we have $\cJ[f](0) \equiv 1$, hence
	\begin{align} \nonumber
		\lowDs \cA[f](0^{+}) &= \liminf_{s\to 0^{+}} \int_{\spt(f)} s^{-1} \big(\cJ[f](x,s) - \cJ[f](x,0)\big) d x \\\label{eq:1}
		&=  \liminf_{s\to 0^{+}}\int_{\spt(f)} s^{-1}\big(\cJ[f](x,s) - 1\big) \, dx \, .
	\end{align}
	Now, \eqref{eq:J^2} guarantees that
	\begin{equation*}
		\cJ[f](x,s)^2 = 1 + s \, G \, , \qquad \text{where}\quad G := \left( |\nabla f(x)|^2 + \dfrac{2 \, \lambda \, f(x) \, \de_{n} f(x)}{\sqrt{|x'|^2 + s f(x)^2}} + \dfrac{R[f](x,\sqrt{s})}{s} \right) \, .
	\end{equation*}
	Thus by \eqref{eq:1} we infer
	\begin{equation*}
		\int_{\spt(f)} s^{-1}(\cJ[f](x,s) - 1)\, dx = \cI^1_s + \cI^2_s \,,
	\end{equation*}
	where
	\begin{equation*}
		\cI^1_s := \dfrac{1}{s} \int_{\spt(f)} \sqrt{1 + s G} - \left( 1 + s \, \dfrac{G}{2} \right) dx \, , \qquad \cI^2_s := \int_{\spt(f)} \dfrac{G}{2} \, dx \, .
	\end{equation*}
	An elementary computation yields
	\begin{equation*}
		\sqrt{1 + s G} - \left( 1 + s \, \dfrac{G}{2} \right) = \dfrac{- s^2 G^2/4}{\sqrt{1 + s G} + \left( 1 + s \dfrac{G}{2} \right)} \, ,
	\end{equation*}
	consequently
	\begin{equation*}
		|\cI^1_s| \leq \dfrac{1}{4} \bigintss_{\spt(f)} \dfrac{s G^2}{\left| \sqrt{1 + s G} + \left( 1 + s \dfrac{G}{2} \right) \right|} dx \, .
	\end{equation*}
	Now it is immediate to observe that $G=0$ whenever $x \notin \spt(f)$, and, as $s\to 0^{+}$, 
	\begin{equation*}
		\sqrt{s} \, G \longrightarrow 0\quad \text{a.e. in $\Sigma$,} \qquad |\sqrt{s} \, G| \leq C \quad \text{for some $C > 0$ depending on $\Lip(f)$ only.}
	\end{equation*}
	By Dominated Convergence, we get
	\begin{equation*}
		\lim_{s \rightarrow 0^+} \cI^1_s = 0 \, .
	\end{equation*}
	Let us now study the limit as $s \rightarrow 0^+$ of
	\begin{align*}
		\cI^2_s & = \int_{\Sigma} \left(\dfrac{1}{2} \, |\nabla f(x)|^2 + \dfrac{\lambda \, f(x) \, \de_{n}f(x)}{\sqrt{|x'|^2 + s f(x)^2}} + \dfrac{1}{2} \, \dfrac{R[f](x,\sqrt{s})}{s}\right) \, d x \\
	\end{align*}
	By \eqref{eq:R_tprop} and Dominated Convergence, we have
	\begin{equation*}
		\int_{\Sigma} \dfrac{R[f](x,\sqrt{s})}{s} \, dx \longrightarrow 0 \, , \qquad \text{as $s \rightarrow 0^+ \, .$}
	\end{equation*}
	On the other hand, denoting by $g(x') := f(x',\lambda|x'|)^2$, an integration by parts yields
	\begin{align*}
		\int_{\Sigma} \dfrac{\lambda \, f(x) \, \de_{n}f(x)}{\sqrt{|x'|^2 + s f(x)^2}} \, d x & = \dfrac{\lambda}{s} \int_{\Sigma} \dfrac{s \, f(s) \, \de_{n}f(x)}{\sqrt{|x'|^2 + s f(x)^2}} \, d x \\
		& = \dfrac{\lambda}{s} \int_{\Sigma} \de_{n} \sqrt{|x'|^2 + s f(x)^2} \, dx \\
		& = \dfrac{\lambda}{s} \int_{\R^{n-1}} \left( \int_{\lambda |x'|}^{+ \infty} \de_{n} \sqrt{|x'|^2 + s f(x)^2} d x_{n} \right) dx' \\
		& = \dfrac{\lambda}{s} \int_{\R^{n-1}} |x'| - \sqrt{|x'|^2 + s g(x')} \, dx' \\
		& = - \lambda \int_{\R^{n-1}} \dfrac{g(x')}{|x'| + \sqrt{|x'|^2 + s g(x')}} \, d x'
	\end{align*}
	Now, the integrand
	\begin{equation*}
		\dfrac{g(x')}{|x'| + \sqrt{|x'|^2 + s g(x')}}
	\end{equation*}
	is positive and monotonically non-increasing in $s$, for any $x' \neq 0$. Then, by Beppo Levi's Theorem, we obtain
	\begin{equation*}
		\lim_{s \to 0^+} \int_{\R^{n-1}} \dfrac{g(x')}{|x'| + \sqrt{|x'|^2 + s g(x')}} \, d x' = \dfrac{1}{2} \int_{\R^{n-1}} \dfrac{g(x')}{|x'|} \, dx' \, .
	\end{equation*}
	This implies that
	\begin{equation*}
		\lim_{s \to 0^+} \cI^2_s = \dfrac{1}{2} \int_{\Sigma} |\nabla f(x)|^2 \, dx - \dfrac{\lambda}{2} \int_{\R^{n-1}} \dfrac{g(x')}{|x'|} \, dx'\, ,
	\end{equation*}
	and the proof is concluded.
\end{proof}


Theorem \ref{thm:d+} allows us to test in which cases $\Sigma$ is stable. Indeed, by \eqref{eq:d+A}, the first variation of the area vanishes, hence the stability condition can be stated in terms of the lower-right second variation:
\begin{equation*}
	\lowDd A[f](0^{+}) = 2\lowDs\cA[f](0^{+}) \geq 0\,,
\end{equation*}
which is equivalent to the functional inequality
\begin{equation}\label{eq:funcineq}
	\int_{\Sigma} |\nabla f(x)|^2 \, dx \geq \lambda \int_{\R^{n-1}} \dfrac{f(x',\lambda|x'|)^2}{|x'|} \, dx' \, .
\end{equation}
\subsection{The case $\dim \Sigma = 2$}

In this case, we see that $\Sigma$ is always unstable (we proved a more general version of this result in \cite{LeoVia1-2024}). 
\begin{theorem} \label{thm:d+<0}
	Let $n = 2$. Then, for all $f \in C^{0,1}_c(\overline{\Sigma})$ with the property that $f(0) \neq 0$, we have
	\begin{equation*}
		\lowDs\cA[f](0^{+}) < 0 \, .
	\end{equation*}
	In particular, this shows that $\Sigma$ is unstable.
\end{theorem}
\begin{proof}
	When $n=2$ the inequality \eqref{eq:funcineq} becomes
	\begin{equation} \label{eq:d+A3}
		\int_{\Sigma} |\nabla f(x)|^2 \, dx \geq \lambda \int_{\R} \dfrac{f(x_1,\lambda|x_1|)^2}{|x_1|} \, d x_1 \, .
	\end{equation}
	We now observe that, for all $f \in C^{0,1}_c(\overline{\Sigma})$, the left-hand side of \eqref{eq:d+A3} is always finite while the right-hand side is finite only if $f(0) = 0$. Thus \eqref{eq:d+A3} fails for all $f \in C^{0,1}_c(\overline{\Sigma})$ such that $f(0) \neq 0$.
\end{proof}

\subsection{The case $\dim \Sigma \geq 3$}
In this case, the aperture of the cone $\Omega_{\lambda}$ plays a role. Indeed, we can prove that there exists a threshold aperture parameter $\lambda^{*}>0$, such that for larger apertures (i.e., for $0<\lambda\le \lambda^{*}$), the hyperplane $\Sigma$ becomes strictly stable. 
\begin{theorem} \label{thm:stab4}
	Let $n \geq 3$. Then there exists $\lambda^{\ast} = \lambda^\ast(n) > 0$ such that, if $0<\lambda \leq \lambda^\ast$, we have
	\begin{equation} \label{eq:d+4}
		\lowDs \cA[f](0^{+}) \geq 0 \qquad \text{for all $f \in  \, C^{0,1}_c(\overline{\Sigma})$}
	\end{equation}
	and the inequality is strict whenever $f$ is not identically zero, which means that $\Sigma$ is strictly stable.
\end{theorem}
To prove Theorem \ref{thm:stab4}, we need the following result.
\begin{lemma} \label{lem:Kato}
	Let $n \geq 3$. Then
	\begin{equation} \label{eq:Kato}
		\int_{\Sigma} |\nabla f(x)|^2 \, dx \geq \dfrac{K_{n}}{(1 + \lambda)^2} \int_{\R^{n-1}} \dfrac{f(x',\lambda|x'|)^2}{|x'|} \, dx' \, , \qquad \text{for all $f \in H^1(\Sigma)$,}
	\end{equation}
	where $K_{n} = 2 \Gamma^{2}(n/4) \Gamma^{-2}((n - 2)/4)$ and $\Gamma(u)$ is Euler's Gamma function.
\end{lemma}
Inequality \eqref{eq:Kato} is known as \emph{Kato's Inequality}. A proof of a stronger version of \eqref{eq:Kato}, valid when $\Sigma$ is a half-space (i.e. for $\lambda = 0$), is provided in \cite[Theorem 1.4]{DavilaDupaigneMontenegro2008}. Although our domain $\Sigma$ is not a half-space, our version directly follows from inequality (9) in the same work. We also mention that a former proof of Kato's Inequality was given by Herbst in \cite{Herbst1977}.
\begin{proof}[Proof of Lemma \ref{lem:Kato}]
	Let us denote by $\R^n_+$ the half-space of points $x \in \R^n$ with $x_{n} > 0$. Inequality (9) of \cite{DavilaDupaigneMontenegro2008} states that
	\begin{equation} \label{eq:Katohalfspace}
		\int_{\R^{n}_+} |\nabla g(x)|^2 \, dx \geq K_{n} \, \int_{\R^{n-1}} \dfrac{g(x',0)^2}{|x'|} \, dx' \, , \qquad \text{for all $g \in H^1(\R^n_+)$.}
	\end{equation}
	Given $f \in H^1(\Sigma)$, let us set
	\begin{equation*}
		T(x) := (x', x_{n} + \lambda |x'|),\qquad g(x) := f(T(x))\,.
	\end{equation*}
	We have $g \in H^1(\R^n_+)$. Since $g(x',0) = f(x',\lambda |x'|)$, we get
	\begin{equation} \label{eq:boundaryterm}
		\int_{\R^{n-1}} \dfrac{g(x',0)^2}{|x'|} \, dx' = \int_{\R^{n-1}} \dfrac{f(x',\lambda |x'|)^2}{|x'|} \, dx' \, .
	\end{equation}
	We observe that
	\begin{equation*}
		\nabla g(x) = \nabla f (T(x)) + \lambda \de_{n} f(T(x)) \left( \dfrac{x'}{|x'|}, 0 \right) \, ,
	\end{equation*}
	and thus
	\begin{equation} \label{eq:estnablag}
		|\nabla g (x)|^2 \leq (1 + \lambda)^2 |\nabla f(T(x))|^2 \, .
	\end{equation}
	Combining \eqref{eq:Katohalfspace}, \eqref{eq:boundaryterm} and \eqref{eq:estnablag}, and noting that $\det DT \equiv 1$, by the change of variable formula we get
	\begin{align*}
		(1 + \lambda)^2 \int_{\Sigma} |\nabla f(x)|^2 dx & = (1 + \lambda)^2 \int_{\R^n_+} |\nabla f(T(x))|^2 dx \\
		& \geq \int_{\R^n_+} |\nabla g(x)|^2 dx \\
		& \geq K_{n} \int_{\R^{n-1}} \dfrac{f(x',\lambda |x'|)^2}{|x'|} \, d x' 
	\end{align*}
	and conclude the proof.
\end{proof}

\begin{proof}[Proof of Theorem \ref{thm:stab4}]
	By Lemma \ref{lem:Kato}, the validity of \eqref{eq:funcineq}, for all $f \in C^{0,1}_{c}(\overline{\Sigma})$, is guaranteed by the condition
	\begin{equation}\label{eq:conditionKlambda}
		0<\lambda \leq \dfrac{K_{n}}{(1 + \lambda)^2} \,.
	\end{equation}
	Since the function $\lambda \mapsto \lambda(1 + \lambda)^2$ is monotonically increasing from $0$ to $+\infty$, there exists a unique $\lambda^\ast > 0$ such that
	\begin{equation*}
		\lambda^\ast(1 + \lambda^\ast)^2 = K_{n} 
	\end{equation*}
	and, consequently, \eqref{eq:conditionKlambda} is satisfied if and only if $0<\lambda \le \lambda^{*}$, as wanted.
\end{proof}

\subsection*{Data Availability Statement}
Data sharing not applicable to this article as no datasets were generated or analyzed during the current study.
\subsection*{Conflict of interest}
The authors declare that they have no financial interests or conflicts of interest related to the subject matter or materials discussed in this paper.


\end{document}